\documentclass[12pt]{article}
\usepackage{amsmath,amssymb,amsthm,amsfonts}
\usepackage{hyperref}
\usepackage{tikz}
\usepackage{enumerate}
\makeatletter
\def\imod#1{\allowbreak\mkern10mu({\operator@font mod}\,\,#1)}
\makeatother

\newtheorem{theorem}{Theorem}[section]
\newtheorem{proposition}{Proposition}[section]
\newtheorem{lemma}{Lemma}[section]
\newtheorem{corollary}{Corollary}[section]

\theoremstyle{definition}
\newtheorem{definition}{Definition}[section]

\allowdisplaybreaks

\usepackage{lscape}
\usepackage{graphicx}
\usepackage{placeins}
\usepackage{float}
\usepackage{subfigure}
\usepackage{color}
\usepackage{wrapfig}  
\usepackage{float}   
\usepackage{lastpage}  
\usepackage{tkz-graph}
\usepackage{tikz-qtree}

\usepackage{graphicx}
	\graphicspath{{./figures/}}
\title{A Constructive Lower Bound on Szemer\'edi's Theorem}

\author{Vladislav Taranchuk\thanks{Dept.\ of Mathematics and Statistics, Sacramento State University, {\tt vtaranchuk@csus.edu}} \and 
}

\begin{document}
\maketitle

\begin{abstract}
Let $r_k(n)$ denote the maximum cardinality of a set $A \subset \{1,2, \dots, n \}$ such that $A$ does not contain a $k$-term arithmetic progression. In this paper, we give a method of constructing such a set and prove the lower bound $n^{1-\frac{c_k}{k \ln k}} < r_k(n)$ where $k$ is prime, and $c_k \rightarrow 1$ as $k \rightarrow \infty$. This bound is the best known for an  increasingly large interval of $n$ as we choose larger and larger $k$.  We also demonstrate that one can prove or disprove a conjecture of Erd\H{o}s on arithmetic progressions in large sets once tight enough bounds on $r_k(n)$ are obtained.
\end{abstract}

\section{Introduction}
In 1927, van der Waerden \cite{VDW} proved that for all positive integers $r$ and $k$, there exists a positive integer $N$ such that every $r$-partitioning of $\{1,2, \dots, N  \}$ contains a $k$-term arithmetic progression in one of the parts of the partition. The smallest such integer $N$ is denoted $w(r, k)$. Van der Waerden's theorem has become the cornerstone of what is known today as the study of Ramsey theory on the integers. 


In 1936, Erd\H{o}s and Tur\'an \cite{Erdos} conjectured that every set of consecutive integers with a positive natural density contains arbitrarily long arithmetic progression. An equivalent restatement of the conjecture claims that every set that does not contain arbitrarily long arithmetic progressions has a natural density that must converge to zero. 
Let $r_k(n)$ denote the maximum cardinality of a set $A \subset \{1,2, \dots, n \}$ such that $A$ does not contain a $k$-term arithmetic progression. Proving that for some positive integer $k$, $\frac{r_k(n)}{n} \rightarrow 0$ as $n \rightarrow \infty$ would prove the conjecture of Erd\H{o}s and Tur\'an to be true. 

In 1953, Klaus Roth \cite{Klaus} made the first step towards solving this conjecture as he proved the upper bound
$$
r_3(n) < \frac{cn}{\log \log n}.
$$
This was followed by an upper bound on $r_4(n)$ obtained by Endre Szemer\'edi \cite{Sz69} in 1969. In 1975, Szemer\'edi \cite{Sz75} proved the conjecture to be true in general by a clever extension of his proof for the case $k= 4$. 

To this day the true growth rate of $r_k(n)$ is still unknown and the problem is still widely studied. Obtaining the true growth rate on $r_k(n)$ has important consequences as it can be used to prove or disprove a famous conjecture by Erd\H{o}s \cite{Erdos} on arithemtic progressions. This conjecture states that if the sum of the reciprocals of a subset of the natural numbers diverges, then the set contains arbitrarily long arithmetic progressions. This conjecture has become of great interest in recent years as Ben Green and Terrence Tao \cite{primes} proved a special case of this conjecture by showing that the primes contain arbitrarily long arithmetic progressions. 

Currently, the best known general upper bound on $r_k(n)$ is due to Gowers \cite{Gowers}. In 2001, Gowers used fourier analysis and combinatorics to prove that 
$$
r_k(n) < \frac{n}{(\log \log n)^{2^{2^{k+9}}}}.
$$
The best known general lower bound is given by Kevin O'Bryant who built upon earlier results by Behrend \cite{Behr}, Rankin \cite{Rank}, and Elkin \cite{elk}. In 2011, O'Bryant \cite{Kevin} proved that 
$$
\frac{cn}{2^{a2^{(a-1)/2}\sqrt[a]{\log n} - \frac{\log \log n}{2a}}} < r_k(n)
$$
where $a = \lceil \log k \rceil$. In this paper, we provide a new recursive construction that gives a lower bound on $r_k(n)$. More specifically, in Section 2 we prove that if $k$ is a prime, then
$$
r_k(n)(k-1) \leq r_k(nk).
$$
We then use this to obtain the following theorem.

\noindent{\bf Theorem \ref{mainthem}.} \
If $n$ is positive integer and $k$ is a prime, then 
$$
n^{1-\frac{c_k}{k \ln k}} < r_k(n)
$$
where $c_k \rightarrow 1$ as $k \rightarrow \infty$.

We obtain this bound by modifying a construction by Blankenship, Cummings, and Taranchuk \cite{me} that was used to prove a recursive lower bound on the van der Waerden numbers. In particular they proved that if $p$ is a prime and $p\leq k$ then 
$$
 w(r, k) > p \cdot \left(w\left(r - \left\lceil \frac{r}{p}\right\rceil, k\right) -1\right).
$$
Our theorem provides the best known bound for  $n < ck^{k^{3/2} \log k}$, however O'Bryant's bound is better as $n \rightarrow \infty$. 

In recent years, extensive research of cases $k =3$ and $k=4$ has yielded tighter bounds on $r_k(n)$.
The case $k=3$ has been of particular interest and the bounds on $r_3(n)$ have seen steady, incremental improvements through the years. Currently, the best known bounds are
$$
\frac{n}{2^{\sqrt{8\log n}}} < r_3(n) < \frac{cn(\log \log n)^4}{\log n}.
$$
The lower bound is due to O'Bryant \cite{Kevin} and the upper bound was given by Thomas Bloom \cite{Bloom} in 2016. In 2017, Ben Green and Terrence Tao \cite{r4} provided the upper bound
$$
r_4(n) < \frac{c_1n}{(\log n)^{c_2}}
$$
for absolute constants $c_1$ and $c_2$.




In the following section we prove our main theorem and in Section 4 we discuss a potential method for improving our bound.

\section{Proof of the Main Theorem}

\begin{definition}
Here are some useful definitions.
\begin{enumerate}
\item The acronym $k$-AP stands for $k$-term arithmetic progression.
\item Let $a_1, a_2, \dots, a_k$ be an arithmetic progression. Then we call the difference between any two consecutive elements, $a_i -a_{i-1} =d$ the \textbf{common difference}.
\end{enumerate}
\end{definition}

\begin{definition}\label{Construct}
Let $k$ be a prime and $A$ be any $k$-AP free subset of $\{ 1, 2, \dots, n  \}$. Define \newline $A_k = \{(a-1)k +1, (a-1)k +2, \dots, (a-1)k + (k-1): a \in A \}$.
\end{definition}

As a quick example, let $k = 3$ and
$$
A = \{ 1, 3, 4, 6\}.
$$
Then
$$
A_3 = \{1,2,7,8,10,11,16,17 \}.
$$
Note that this method of construction is equivalent to replacing each term in $A$ with a consecutive sequence of length $k$, and then excluding the last element. This excluded element is always a multiple of $k$ and thus, $A_k$ contains no elements that are congruent to 0(mod $k$).
This definition is the basis for our recursive construction. Our next lemma will prove a key property about $A_k$.

\begin{lemma}\label{expand}
If $A$ is any $k$-AP free set, then $A_k$ is also $k$-AP free.
\end{lemma}

\begin{proof}
We split this proof into two cases. 

\noindent \emph{Case 1}: Assume there exists a $k$-AP in $A_k$ such that the common difference of this $k$-AP is $d$, and $d \equiv 0$(mod $k$). This implies that for some positive integer $m$, $d = mk$.
Such an AP would be representative of finding an AP with common difference $m$ in the original set $A$, which we defined as being $k$-AP free. This is due to the expansion of each term in $A$ to a block of $k-1$ elements in $A_k$. Thus, this is a contradiction. 

\noindent \emph{Case 2}: Assume there exists a $k$-AP in $A_k$ such that the common difference of this $k$-AP is $d$, and $d \not\equiv 0$(mod $k$). Recall that since $k$ is a prime and if $i$ is an arbitrary integer such that $i \not\equiv 0$(mod $k$), then $i \not\equiv 2i \not\equiv \dots \not\equiv ki$(mod $k$). This is important because it implies that multiples of $i$ must first cycle through all possible congruence classes mod $k$. Similarly, if the common difference $d$ of a $k$-AP is not congruent to 0(mod $k$), then each element in the $k$-AP is in a unique congruence class mod $k$. More importantly, there must be an element in every congruence class of $k$ for such a $k$-AP. However, note that the definition of $A_k$ excludes elements congruent to 0(mod $k$). Thus, this would also be a contradiction. So $A_k$ is $k$-AP free.
\end{proof}

\begin{corollary}\label{recursive}
If $n$ is a positive integer and $k$ is a prime, then
$$
r_k(n)(k-1) \leq r_k(nk).
$$
\end{corollary}

\begin{proof}
Recall that in Definition \ref{Construct}, we replace each element in $A$ with $k-1$. This implies that $|A_k| = |A|(k-1)$. By Lemma \ref{expand}, we have that if $A$ is $k$-AP free then so is $A_k$. Thus, if we set $|A| = r_k(n)$, then $|A_k| = r_k(n)(k-1) \leq r_k(nk)$.
\end{proof}

\begin{theorem}\label{mainthem}
If $n$ is a positive integer and $k$ is a prime, then 
$$
n^{1-\frac{c_k}{k \ln k}} < r_k(n)
$$
where $c_k \rightarrow 1$ as $k \rightarrow \infty$.
\end{theorem}

\begin{proof}
It is known that $k-1 = r_k(k)$. By Corollary \ref{recursive} we obtain that $(k-1)^2 \leq r_k(k^2)$. Continuing this recursive process of construction we obtain that for any positive integer $r$, $(k-1)^r \leq r_k(k^r)$. 

Let $n = k^r$, so $r= \log_k(n) = \frac{\ln n}{\ln k}$. Define $f_k(n) = (k-1)^r < r_k(n)$. Considering the ratio of $n$ over $f_k(n)$, we obtain that 
$$
\frac{n}{f_k(n)} = \left(\frac{k}{k-1}\right)^r = \left(\frac{k}{k-1}\right)^{k \cdot \frac{r}{k}} = e ^{\left( c_k \cdot \frac{r}{k}\right)}
$$
where $c_k = \ln \left( (\frac{k}{k-1})^k \right)$. Note that $c_k \rightarrow 1$ as $k \rightarrow \infty$, since
$$
\lim_{k\to\infty} \left(\frac{k}{k-1}\right)^k = e.
$$
Thus, we obtain that 
$$
\frac{n}{f_k(n)} = e ^{\left( c_k \cdot \frac{r}{k}\right)} = n^{\left(\frac{c_k}{k \ln k}\right)} \implies 
f_k(n) = n^{1-\frac{c_k}{k \ln k}}.
$$
If $k^{r-1} < n < k^r$, then we can obtain the construction of size $k^r$ and only consider the elements up to $n$. 
\end{proof}

By comparing the bound obtained by Theorem \ref{mainthem} and O'Bryant's bound, it is easy to check that Theorem \ref{mainthem} is the best known lower bound on $r_k(n)$ for all $n < ck^{k^{3/2}\log k}$.  

\section{Erd\H{o}s's Conjecture and Szemer\'edi's Theorem}
A set of positive integers is called \textbf{large} if the sum of the reciprocals of its elements diverges, otherwise the set is called \textbf{small}. Erd\H{o}s conjectured that every large set contains arbitrarily long arithmetic progressions. The contrapositive of the conjecture is as follows; if there exists a $k$ such that a set $A$ does not contain a $k$-AP, then $A$ is small. The only progress to be made on this conjecture was made by Green and Tao(cite), who proved that the primes contain arbitrarily long arithmetic progressions. Although a relationship between Szemer\'edi's theorem and this conjecture is clear, there has been no explicit connection made. 
In this section we show an intimate relationship between these ideas.

\begin{lemma}\label{inverses}
If $f(n)$, $g(n)$ and $h(n)$ are unbounded monotonically increasing functions with
$$
f(n) < g(n) < h(n)
$$
then 
$$
f^{-1}(n) > g^{-1}(n) > h^{-1}(n).
$$
\end{lemma}
\begin{proof}
Assume that $f(n)$, $g(n)$, and $h(n)$ are unbounded monotonically increasing functions with $ f(n) < g(n) < h(n) $. Note that this implies that $f^{-1}(n)$, $g^{-1}(n)$, and $h^{-1}(n)$ are also unbounded monotonically increasing functions. It is clear that 
$$
f^{-1}(f(n)) = g^{-1}(g(n)) = h^{-1}(h(n)).
$$
However, since $f(n) < g(n) < h(n)$, then we obtain that 
$$
f^{-1}(f(n)) > g^{-1}(f(n)) > h^{-1}(f(n))
$$
which implies that
$$
f^{-1}(n) > g^{-1}(n) > h^{-1}(n).
$$
\end{proof}
Using this lemma, we will prove the following proposition which is the key concept behind proving or disproving Erd\H{o}s's conjecture.
\begin{proposition}\label{bounds}
If $A \subset \mathbb{N}$ is an infinite set where $f(n)$ and $h(n)$ are respective lower and upper bounds on the cardinality of $A$ up to $n$, then $f^{-1}(a)$ and $h^{-1}(a)$ are respective upper and lower bounds on the size of the $a^{th}$ element in $A$.
\end{proposition}
\begin{proof}
Assume $A \subset \mathbb{N}$ is an infinite set with $f(n)$ and $h(n)$ as lower and upper bounds on $g(n)$, respectively. Define $g(n)$ to be the true growth rate of the cardinality of $A$. Note that $g(n) = a$ implies that $A$ contains approximately $a$ elements up to $n$. We then bring the readers attention to the fact that $g^{-1}(a) = n$. This implies that the first $a$ elements in $A$ are from the subset of $\{ 1, 2, \dots , n  \}$. In other words, $g^{-1}(a) = n$ implies that the $a^{th}$ element is approximately $n$. We now apply Lemma \ref{inverses} to see that $f^{-1}(a) > g^{-1}(a) > h^{-1}(a)$, which implies that $f^{-1}(a)$ is an upper bound on the $a^{th}$ element in $A$, and likewise, $h^{-1}(a)$ is a lower bound on the $a^{th}$ element in $A$.
\end{proof}
As an example of this proposition we give the following example. Let $A = \{1, 2, 4, 5, 9, 14, 16, \dots \}$ be a random subset of $\mathbb{N}$. Then $g(14) = 6$, and let $f(14) = 5$ and $h(14) = 7$. Then by Lemma  \ref{inverses} we have that 
$$f^{-1}(6) > 14 = f^{-1}(5) = g^{-1}(6) = h^{-1}(7) = 14 > h^{-1}(6).$$ 
\begin{corollary}
If $f_k(n)$ and $h_k(n)$ are lower and upper bounds on $r_k(n)$ respectively, and $A$ is a set whose density is defined by $r_k(n)$, then 
$$
\sum_{n=1}^{\infty} \frac{1}{f^{-1}_k(n)} < \sum_{n=1}^{\infty} \frac{1}{a_n} < \sum_{n=1}^{\infty} \frac{1}{h^{-1}_k(n)} 
$$
are bounds on the sum of the reciprocals of the elements in $A$, where $a_n$ is the $n^{th}$ element in $A$. 
\end{corollary}

\begin{proof}
Proposition \ref{bounds} gives us that since $f_k(n)$ and $h_k(n)$ are lower and upper bounds on $r_k(n)$ respectively, then $f^{-1}_k(n)$ and $h^{-1}_k(n)$ are upper and lower bounds on the size of the $n^{th}$ element in a set $A \subset \mathbb{N}$ that does not contain a $k$-AP, then  
$$
f^{-1}_k(n) > a_n > h^{-1}(n) \implies \frac{1}{f^{-1}_k(n)} < \frac{1}{a_n} < \frac{1}{h^{-1}_k(n)}
$$
which consequently implies
$$
\sum_{n=1}^{\infty} \frac{1}{f^{-1}_k(n)} < \sum_{n=1}^{\infty} \frac{1}{a_n} < \sum_{n=1}^{\infty} \frac{1}{h^{-1}_k(n)}. 
$$
\end{proof}
Thus, the key to proving or disproving Erd\H{o}s's conjecture is in studying and tightening the bounds on $r_k(n)$. At this point some analytic tools will help us identify the growth rate of $r_k(n)$ required to prove or disprove the conjecture.

The following result is not difficult to show using Cauchy's Condensation test, Hardy proves it in his textbook \emph{Course of Pure Mathematics} on page 376. The result states that for a large enough $N \in \mathbb{N}$,
$$
\sum_{n=N}^{\infty}\frac{1}{n(\ln n)(\ln \ln n)\dots(\underbrace{\ln \ln \dots \ln n}_{d\text{ times}})^s}
$$
converges for $s>1$ and diverges otherwise. The following two theorems hold true for defined bounds in theorem. However, it is clear that both cannot be true, thus only one of the theorems hold true in relation to $r_k(n)$. For more detailed results on converging and diverging series as related to densities we refer the reader to a recent paper by Niculescu and Pr\v{a}jitur\v{a}(cite).
\begin{theorem}\label{False}
There exists positive integers $d$, $k$, and a constant $c$ for which,
$$
\frac{c \cdot n}{(\ln n)(\ln \ln n)\dots(\underbrace{\ln \ln \dots \ln n}_{d\text{ times}})} \leq r_k(n)
$$
if and only if there exists a large set $A$ that does not contain $k$-APs. 
\end{theorem}
\begin{proof}
Assume that $f_k(n)$ is a lower bound for $r_k(n)$ such that
$$
f_k(n) = \frac{c \cdot n}{(\ln n)(\ln \ln n)\dots(\underbrace{\ln \ln \dots \ln n}_{d\text{ times}})}.
$$
Obtaining an exact $f^{-1}_k(n)$ solely as a function of $n$ is not possible. However, if for some positive integer $k$, there exists an $n \in \mathbb{N}$ and a function $g(n)$ such that $f_k(g(n)) > n$, for all $n > N$, then it is clear that $g(n) > f^{-1}_k(n)$ for all $n > N$. This implies that $g(n)$ is a worse upper bound on the $n^{th}$ element in a set $A$ whose density is defined by $r_k(n)$. Consider
$$
g(n) = n(\ln n)(\ln \ln n)\dots(\underbrace{\ln \ln \dots \ln n}_{(d+1)\text{ times}}).
$$ 
Then we obtain that 
$$
f_k(g(n)) = \frac{c \cdot g(n)}{(\ln g(n))(\ln \ln g(n))\dots(\underbrace{\ln \ln \dots \ln g(n)}_{d\text{ times}})}.
$$
We can show there exists an $N \in \mathbb{N}$ such that $n < f_k(g(n))$ for all $n > N$, by considering that $g(n) < n^2$ and using this for our denominator in $f_k(g(n))$.
$$
\frac{c \cdot g(n)}{(\ln g(n))(\ln \ln g(n))\dots(\underbrace{\ln \ln \dots \ln g(n)}_{d\text{ times}})} > \frac{c \cdot g(n)}{(\ln n^2)(\ln \ln n^2)\dots(\underbrace{\ln \ln \dots \ln n^2}_{d\text{ times}})}.
$$
Then note that there exists an absolute constant $C$ such that
$$
\frac{c \cdot g(n)}{(\ln n^2)(\ln \ln n^2)\dots(\underbrace{\ln \ln \dots \ln n^2}_{d\text{ times}})} > \frac{c \cdot g(n)}{C \cdot (\ln n)(\ln \ln n)\dots(\underbrace{\ln \ln \dots \ln n}_{d\text{ times}})}.
$$
Recall that we defined $g(n)$ to have $d+1$ iterations of natural log terms. Thus we obtain that 
$$
\frac{c \cdot g(n)}{C \cdot (\ln n)(\ln \ln n)\dots(\underbrace{\ln \ln \dots \ln n}_{d\text{ times}})} = \frac{n(\overbrace{\ln \ln \dots \ln n}^{(d+1)\text{ times}})}{c}
$$
for some new constant $c$. Since the constant does not grow, then we have that there exists an $N \in \mathbb{N}$, such that $f_k(g(n)) > n$ for all $n > N$. This implies that $f^{-1}_k(n) > g(n)$, which further implies that $\frac{1}{g(n)} <\frac{1}{f^{-1}_k(n)}$, for all $n > N$. Thus we obtain that
$$
\sum_{n=N}^{\infty} \frac{1}{g(n)} =  \sum_{n=N}^{\infty} \frac{1}{n(\ln n)(\ln \ln n)\dots(\underbrace{\ln \ln \dots \ln n}_{d\text{ times}})} < \sum_{n=N}^{\infty} \frac{1}{f^{-1}_k(n)}.
$$
Relating back to Cauchy's condensation test, we have that $\sum_{n=N}^{\infty} \frac{1}{g(n)}$ diverges. Thus, $\sum_{n=N}^{\infty} \frac{1}{f^{-1}_k(n)}$ diverges. Thus, for some $k$ there exists large set $A$ that contains no $k$-AP. 

Having concluded the forwards direction of the proof, note that the backwards implication of the proof which assumes that for some $k$, there exists a large set not containing a $k$-AP, becomes an easy consequence of working back from Cauchy's Condensation test. The existence of such an $A$ implies that the growth rate of $r_k(n)$ must atleast be that of our defined $f_k(n)$ for some fixed $d$.  
\end{proof}
\begin{corollary}
If there exists positive integers $d$, $k$, and a constant $c$ for which,
$$
\frac{c \cdot n}{(\ln n)(\ln \ln n)\dots(\underbrace{\ln \ln \dots \ln n}_{d\text{ times}})} \leq r_k(n)
$$
then Erd\H{o}s's conjecture on large sets and arithmetic progressions is false.
\end{corollary}
\begin{proof}
If such a lower bound for $r_k(n)$ exists, then we use Theorem \ref{False} to claim that there exists a large set $A$ that does not contain $k$-APs. This is contradictory to the contrapositive statement of Erd\H{o}s's conjecture, since such an $A$ is not small.
\end{proof}
\begin{theorem}\label{True}
For all positive integers $k$, there exists a positive integer $d$, a constant $c$, and $s > 1$ for which,
$$
r_k(n) \leq \frac{c \cdot n}{(\ln n)(\ln \ln n)\dots(\underbrace{\ln \ln \dots \ln n}_{d\text{ times}})^s}
$$
if and only if every set $A$ that does not contain a $k$-AP is small.
\end{theorem}
\begin{proof}
Assume that for all $k$, $h_k(n)$ is an upper bound on $r_k(n)$ where
$$
h_k(n) = \frac{c \cdot n}{(\ln n)(\ln \ln n)\dots(\underbrace{\ln \ln \dots \ln n}_{d\text{ times}})^s}.
$$
Similarly as from the proof of Theorem \ref{False}, it is not possible to find an exact inverse function for $h_k(n)$ in terms of solely $n$. In this case, we want a $g(n)$ such that for all positive integers $k$, there exists an $N \in \mathbb{N}$, such that $h_k(g(n)) < n$ for all $n > N$. This would imply that $g(n) < h^{-1}_k(n)$ for all $n > N$. Consider
$$
g(n) = n(\ln n)(\ln \ln n)\dots(\underbrace{\ln \ln \dots \ln n}_{d\text{ times}})^{s - \epsilon}
$$ 
for some $\epsilon > 0$ and $s - \epsilon > 1$. Clearly such an epsilon exists since $s > 1$ and $s \in \mathbb{R}$. Then we obtain that 
$$
h_k(g(n)) = \frac{c \cdot g(n)}{(\ln g(n))(\ln \ln g(n))\dots(\underbrace{\ln \ln \dots \ln g(n)}_{d\text{ times}})^s}.
$$
Expanding $g(n)$ as defined gives us,
$$
\frac{c \cdot g(n)}{(\ln g(n))(\ln \ln g(n))\dots(\underbrace{\ln \ln \dots \ln g(n)}_{d\text{ times}})^s} = \frac{c \cdot (n(\ln n)(\ln \ln n)\dots(\overbrace{\ln \ln \dots \ln n}^{d\text{ times}})^{s - \epsilon})}{(\ln g(n))(\ln \ln g(n))\dots(\underbrace{\ln \ln \dots \ln g(n)}_{d\text{ times}})^s}
$$
Clearly, $\ln g(n) > \ln n$, $\ln \ln g(n) > \ln \ln n$, and so on. Also, since our last term is raised to the $s - \epsilon < s$, then we know that for all positive integers $k$ there exists an $N \in \mathbb{N}$, for which $h_k(g(n)) < n$ for all $n > N$.
Thus $g(n) < h^{-1}_k(n)$ for all $n > N$. This further implies that $\frac{1}{h^{-1}_k(n)} < \frac{1}{g(n)}$. Here we note that 
$$
\sum_{n=N}^{\infty} \frac{1}{g(n)} = \frac{1}{n(\ln n)(\ln \ln n)\dots(\underbrace{\ln \ln \dots \ln n}_{d\text{ times}})^{s - \epsilon}} > \sum_{n =N}^{\infty} \frac{1}{h^{-1}_k(n)}.
$$ 
Since $s - \epsilon > 1$, by Cauchy's condensation test we have that $\sum_{n=N}^{\infty} \frac{1}{g(n)}$ converges. Thus $\sum_{n =N}^{\infty} \frac{1}{h^{-1}_k(n)}$ converges, which implies that every set $A$ that does not contain a $k$-AP is small.

The backwards direction of the double implication also comes from the definitions. If every set $A$ that does not contain a $k$-AP for some $k$, is small, then this places an automatic upper bound on $r_k(n)$ given that too large a density creates a large set. Working backwards from Cauchy's condensation test, we again see that the we can obtain that an upper bound on $r_k(n)$ is atleast that of our defined $h_k(n)$ for some positive integer $d$, and real number $s > 1$. 
\end{proof}
\begin{corollary}
If for all positive integers $k$, there exists a positive integer $d$, a constant $c$, and $s > 1$ for which,
$$
r_k(n) \leq \frac{c \cdot n}{(\ln n)(\ln \ln n)\dots(\underbrace{\ln \ln \dots \ln n}_{d\text{ times}})^s}
$$
then Erd\H{o}s's conjecture on large sets and arithmetic progressions is proven true.
\end{corollary}
\begin{proof}
If such an upper bound $r_k(n)$ exists, then we use Theorem \ref{True} to claim that then every set that does not contain a $k$-AP is small. Which is an equivalent statement to that of Erd\H{o}s's conjecture.
\end{proof}
\section{Potential Improvements}
Using the same general idea, we can consider how does using a prime smaller than $k$ affect our construction. It is likely that this would allow us to add elements that would normally be part of the empty blocks between elements in our construction. Although this method would provide a smaller starting bound, it could provide a bound that grows with $n$, which would be an asymptotic improvement over the bound given by Theorem \ref{mainthem}.

\section{Acknowledgments}
I would like to thank Craig Timmons for his insightful comments that helped improve the quality of this paper.

\end{document}